\documentclass[11pt,a4paper,reqno]{amsart}

\usepackage[T1]{fontenc}
\usepackage{lmodern}
\usepackage{microtype}
\usepackage{amsmath,amssymb,mathtools}
\usepackage{aliascnt}
\usepackage{etoolbox}
\usepackage{enumitem}
\usepackage{booktabs}
\usepackage{xcolor}
\usepackage[
  a4paper,
  top=24mm,
  bottom=26mm,
  left=25mm,
  right=25mm,
  includeheadfoot
]{geometry}
\usepackage[hidelinks]{hyperref}

\allowdisplaybreaks
\setlist[enumerate]{leftmargin=*,itemsep=2pt,topsep=4pt}
\setlist[itemize]{leftmargin=*,itemsep=2pt,topsep=4pt}

\newtheorem{theorem}{Theorem}[section]

\newaliascnt{lemma}{theorem}
\newtheorem{lemma}[lemma]{Lemma}
\aliascntresetthe{lemma}

\newaliascnt{proposition}{theorem}

\aliascntresetthe{proposition}

\newaliascnt{corollary}{theorem}
\newtheorem{corollary}[corollary]{Corollary}
\aliascntresetthe{corollary}

\newaliascnt{claim}{theorem}
\newtheorem{claim}[claim]{Claim}
\aliascntresetthe{claim}

\theoremstyle{definition}

\newaliascnt{definition}{theorem}

\aliascntresetthe{definition}

\newaliascnt{problem}{theorem}
\newtheorem{problem}[problem]{Problem}
\aliascntresetthe{problem}

\theoremstyle{remark}

\newaliascnt{remark}{theorem}
\newtheorem{remark}[remark]{Remark}
\aliascntresetthe{remark}

\numberwithin{equation}{section}

\usepackage[nameinlink,capitalise,noabbrev]{cleveref}

\crefname{theorem}{Theorem}{Theorems}
\Crefname{theorem}{Theorem}{Theorems}
\crefname{lemma}{Lemma}{Lemmas}
\Crefname{lemma}{Lemma}{Lemmas}
\crefname{proposition}{Proposition}{Propositions}
\Crefname{proposition}{Proposition}{Propositions}
\crefname{corollary}{Corollary}{Corollaries}
\Crefname{corollary}{Corollary}{Corollaries}
\crefname{claim}{Claim}{Claims}
\Crefname{claim}{Claim}{Claims}
\crefname{definition}{Definition}{Definitions}
\Crefname{definition}{Definition}{Definitions}
\crefname{problem}{Problem}{Problems}
\Crefname{problem}{Problem}{Problems}
\crefname{remark}{Remark}{Remarks}
\Crefname{remark}{Remark}{Remarks}
\crefname{section}{Section}{Sections}
\Crefname{section}{Section}{Sections}
\crefname{subsection}{Section}{Sections}
\Crefname{subsection}{Section}{Sections}
\crefname{subsubsection}{Section}{Sections}
\Crefname{subsubsection}{Section}{Sections}

\newcommand{\qB}{q_B}
\newcommand{\calE}{\mathcal{E}}

\newcommand{\gG}{\gamma_G}
\newcommand{\kappaG}{\kappa_G}

\newcommand{\Dtwo}{\Delta_2}

\title{Degeneracy bounds, stability, and a sharp gap for $B$-colorings}

\author[\NoCaseChange{X. Hu, J. Kong, and Y. Wang}]{%
  \NoCaseChange{%
    \begin{tabular}{@{}c@{}}
      Xiaoxue Hu\textsuperscript{1}
      \
      Jiangxu Kong\textsuperscript{2,*}
      \ and \
      Yiqiao Wang\textsuperscript{3}
      \tabularnewline[4pt]
      {\normalfont\itshape
       \textsuperscript{1}School of Science, Zhejiang University of Science and Technology, Hangzhou 310023, China}
      \tabularnewline[3pt]
      {\normalfont\itshape
       \textsuperscript{2}School of Mathematics, Hangzhou Normal University, Hangzhou, 311121, China}
      \tabularnewline[3pt]
      {\normalfont\itshape
       \textsuperscript{3}Department of Mathematics, Beijing University of Technology, Beijing 100124, China}
      \tabularnewline[4pt]
      {\normalfont
       \textsuperscript{*}Corresponding author. Email: \href{kjx@hznu.edu.cn} {kjx@hznu.edu.cn}}
    \end{tabular}%
  }%
}

\makeatletter

\renewcommand{\section}{%
  \@startsection{section}{1}{\z@}%
    {.7\linespacing\@plus\linespacing}%
    {.5\linespacing}%
    {\normalfont\bfseries\raggedright}}

\renewcommand{\subsection}{%
  \@startsection{subsection}{2}{\z@}%
    {.5\linespacing\@plus.7\linespacing}%
    {.3\linespacing}%
    {\normalfont\bfseries\raggedright}}

\renewcommand{\subsubsection}{%
  \@startsection{subsubsection}{3}{\z@}%
    {.5\linespacing\@plus.7\linespacing}%
    {.3\linespacing}%
    {\normalfont\itshape\raggedright}}

\makeatother

\date{}

\hypersetup{
  pdftitle={Degeneracy Bounds, Stability, and a Sharp Gap for B-Colorings},
  pdfauthor={Xiaoxue Hu, Jiangxu Kong, and Yiqiao Wang},
  pdfsubject={Graph theory; edge-coloring; extremal and stability results},
  pdfkeywords={B-coloring, rainbow 4-cycle, degeneracy, maximum codegree, stability, complete bipartite graph, loopless multigraph.}
}

\begin{document}

\begin{abstract}
A $B$-coloring of a graph is a proper edge-coloring in which every $4$-cycle is rainbow, and $q_B(G)$ denotes the minimum number of colors in such a coloring.  Let $\Delta_2(G)$ denote the maximum number of common neighbors of two distinct vertices of $G$.  We prove that, for integers $1\le d\le\Delta$, every finite simple $d$-degenerate graph $G$ with $\Delta(G)\le\Delta$ satisfies 
$$q_B(G)\le \Delta+(d-1)\Delta_2(G)\le d\Delta.$$ 
Consequently, $d\Delta$ is the exact maximum, with equality precisely for graphs containing $K_{d,\Delta}$.  More generally, if $q_B(G)\ge d\Delta-s$, where $0\le s<\Delta$, then $G$ contains $K_{d,\Delta-s}$; if also $s<d$, then $G$ has at least $d-s$ vertices of degree $\Delta$ with the same open neighborhood.  For $\Delta\ge3$, we further show that every $K_{3,\Delta}$-free 3-degenerate graph satisfies $q_B(G)\le3\Delta-2$; the example $K_{3,\Delta-1}$ shows that this bound is best possible up to one.

For loopless multigraphs, we establish a sharp gap in the possible values of $q_B(G)$.  For every integer $\Delta\ge3$, every finite loopless multigraph $G$ with $\Delta(G)\le\Delta$ satisfies 
$$q_B(G)\le\Delta(\Delta-1)$$ 
unless $G$ has a component isomorphic to $K_{\Delta,\Delta}$, in which case $q_B(G)=\Delta^2$.  The bound $\Delta(\Delta-1)$ is attained by both $K_{\Delta,\Delta-1}$ and $K_{\Delta,\Delta}-e$. Consequently, among finite loopless multigraphs with maximum degree at most $\Delta$, no value of $q_B(G)$ lies strictly between $\Delta^2-\Delta$ and $\Delta^2$.


\end{abstract}

\maketitle

\section{Introduction}\label{sec:intro}

A proper edge-coloring assigns distinct colors to incident edges.  Such
a coloring of a graph $G$ is a \emph{$B$-coloring} if every $4$-cycle
receives four distinct colors.  Let $\qB(G)$ denote the minimum number of
colors in a $B$-coloring of $G$.  The terminology was
introduced by Gy\'arf\'as and S\'ark\"ozy in connection with less strong
edge-colorings and the Brown--Erd\H{o}s--S\'os $(7,4)$ problem
\cite{GyarfasSarkozy2023}.  Rainbow edge-colorings of $4$-cycles had
appeared earlier in the study of hypercubes by Faudree, Gy\'arf\'as,
Lesniak and Schelp \cite{FaudreeEtAl1993}.  This notion is unrelated to
the vertex $b$-chromatic number.

Define the \emph{$B$-conflict graph} $\Gamma_B(G)$ on $E(G)$ by joining
two edges when they are incident or opposite on a $4$-cycle.  Thus
$\qB(G)=\chi\bigl(\Gamma_B(G)\bigr)$.
If $G$ is simple and has maximum degree $\Delta$, every edge has at most
$\Delta^2-1$ $B$-conflicting edges, so the greedy bound gives
$\qB(G)\le\Delta^2$.  The complete bipartite graph
$K_{\Delta,\Delta}$ attains equality, since all of its $\Delta^2$ edges
pairwise conflict.

Previous work has studied the parameter on planar and outerplanar
graphs \cite{GyarfasEtAl2024,KongWangZheng2026}, regular bipartite graphs
\cite{GyarfasSarkozyWagner2026}, and subcubic graphs
\cite{XueHuKong2026}.  For $\Delta\ge3$, Lin and Lin
\cite{LinLin2026} proved the semistrong bound $\Delta^2-1$ for every
connected simple graph of maximum degree at most $\Delta$ other than
$K_{\Delta,\Delta}$.  Since every semistrong edge-coloring is a
$B$-coloring, the same $\Delta^2-1$ bound holds for $B$-colorings.  
For distinct vertices $x,y\in V(G)$, their codegree is
$|N_G(x)\cap N_G(y)|$.  The \emph{maximum codegree} of $G$ is
\[
  \Dtwo(G)=
  \max_{\substack{x,y\in V(G)\\x\ne y}}
  |N_G(x)\cap N_G(y)|.
\]
When $|V(G)|<2$, we set $\Dtwo(G)=0$.  Vuolo \cite{Vuolo2026} used this parameter in
a codegree-sensitive bound for planar graphs.  We use
$\Dtwo(G)$ to bound $q_B(G)$ for arbitrary $d$-degenerate graphs,
without assuming planarity.

Our first results determine the extremal value in terms of degeneracy and maximum codegree and give stability near equality.  Recall that $G$ is \emph{$d$-degenerate} if every nonempty subgraph has a vertex of degree at most $d$.  We use the term \emph{complete bipartite core} for a complete bipartite subgraph, not necessarily induced.

\begin{theorem}
\label{thm:deg-codeg-intro}
\label{thm:deg-codeg}
Let $d$ and $\Delta$ be integers with $1\le d\le\Delta$, and let $G$ be a finite simple $d$-degenerate graph with $\Delta(G)\le\Delta$.  Then $$\qB(G)\le \Delta+(d-1)\Dtwo(G)\le d\Delta.$$
\end{theorem}

The upper bound $d\Delta$ is attained by $K_{d,\Delta}$.  More strongly,
near equality forces a large complete bipartite core.

\begin{theorem}
\label{thm:near-core-intro}
Let $d$, $\Delta$, and $s$ be integers with $1\le d\le\Delta$ and
$0\le s<\Delta$.  Let $G$ be a finite simple $d$-degenerate graph with
$\Delta(G)\le\Delta$.  If $\qB(G)\ge d\Delta-s$, then $G$ contains a
copy of $K_{d,\Delta-s}$.  If in addition $s<d$,
then $G$ also contains a copy of $K_{d-s,\Delta}$.
\end{theorem}

In particular, $\qB(G)=d\Delta$ if and only if $G$ contains a copy of $K_{d,\Delta}$.  Equivalently, this occurs precisely when $G$ has a false-twin class of size at least $d$ whose vertices have degree $\Delta$, where two vertices are \emph{false twins} if they have the same open neighborhood.  Tracking the full deficiency in a degeneracy edge ordering also yields a one-defect description when $\qB(G)\ge d\Delta-1$.

When $d=3$, excluding the extremal core yields a further improvement. For every integer $\Delta\ge3$, each finite simple $3$-degenerate graph $G$ with $\Delta(G)\le\Delta$ and no copy of $K_{3,\Delta}$ satisfies $\qB(G)\le3\Delta-2$, whereas $\qB(K_{3,\Delta-1})=3\Delta-3$.

At the dense endpoint $d=\Delta$, the preceding theorem gives the general extremal value $\Delta^2$ and identifies
$K_{\Delta,\Delta}$ as the extremal core.  After excluding a $K_{\Delta,\Delta}$-component, our second main result strengthens the general simple-graph bound $\Delta^2-1$ to $\Delta(\Delta-1)$ and extends it to loopless multigraphs.

\begin{theorem}[Sharp gap theorem]\label{thm:main}
Let $\Delta\ge3$ be an integer.  If $G$ is a finite loopless multigraph with $\Delta(G)\le\Delta$ and no component isomorphic to the simple graph $K_{\Delta,\Delta}$, then
\[
  \qB(G)\le\Delta(\Delta-1).
\]
\end{theorem}

For $\Delta=3$, \cref{thm:main} follows from a theorem of Xue, Hu, and Kong \cite{XueHuKong2026}.  The bound is attained by $K_{\Delta,\Delta-1}$ and by $K_{\Delta,\Delta}-e$.  Thus no finite loopless multigraph $G$ with $\Delta(G)\le\Delta$ has $\qB(G)$ strictly between $\Delta^2-\Delta$ and $\Delta^2$.

The degeneracy--codegree bound follows from a degeneracy ordering of $G$ and the induced ordering of $E(G)$.  Applying this ordering to vertex-critical subgraphs of $\Gamma_B(G)$ yields the near-equality conclusions and the refined second-level bound for $3$-degenerate graphs. For the gap at the dense endpoint, the proof reduces a minimal counterexample to a connected $\Delta$-regular simple graph. In the bipartite case, K\H{o}nig's line-coloring theorem \cite{BondyMurty2008} decomposes $E(G)$ into at most $\Delta$ matchings, and Brooks' theorem \cite{Brooks1941} colors the auxiliary
conflict graph on each matching with at most $\Delta-1$ colors.  In the nonbipartite case, the weighted quotient by false-twin classes, together with an analysis of a shortest odd cycle, yields a clean endpoint.  Hall's theorem then supplies pairwise distinct available colors for the edges incident with this endpoint, completing the extension.

Throughout, multigraphs are loopless and may have parallel edges.  Degrees in a multigraph are counted with multiplicity, whereas $N_G(v)$ denotes the set of distinct neighbors of $v$.  A $4$-cycle has four distinct vertices.  In a multigraph, different choices of parallel edges may determine distinct $4$-cycles. For distinct vertices $u$ and $v$ of a multigraph $G$, let $\mu_G(u,v)$ denote the number of edges joining them.  For $F\subseteq E(G)$, let $G-F$ denote the spanning subgraph obtained by deleting the edges in $F$.

\section{Degeneracy--codegree bounds and stability}
\label{sec:degenerate}

Throughout this section, all graphs are simple.  The following lemma bounds the later $B$-conflicts in the edge order induced by a degeneracy ordering.

\begin{lemma}\label{lem:later-conflict}
Let $d\ge1$ be an integer, and let $G$ be a $d$-degenerate graph with degeneracy ordering $v_1,\ldots,v_n$.  Order $E(G)$ by increasing index of the earlier endpoint, breaking ties arbitrarily.  For an edge $e=v_i u_1$, where $u_1$ is later than $v_i$, let $u_1,\ldots,u_r$ be the later neighbors of $v_i$.  Then the number of edges following $e$ in this ordering that $B$-conflict with $e$ is at most
\[
  d_G(u_1)-1+
  \sum_{j=2}^{r}|N_G(u_1)\cap N_G(u_j)|.
\]
\end{lemma}

\begin{proof}
Among the edges following $e$, at most $r-1$ are incident with $v_i$, and at most $d_G(u_1)-1$ are incident with $u_1$.

Suppose that an edge $f$ following $e$ is opposite to $e$ on a $4$-cycle $v_i u_1 w z v_i$, where $f=wz$.  If $z$ preceded $v_i$, then $f$ would lie in a block preceding that of $e$.  Hence $z=u_j$ for some $j\in\{2,\ldots,r\}$.  For fixed $j$, the vertex
$w$ is a common neighbor of $u_1$ and $u_j$ other than $v_i$, so there are at most $|N_G(u_1)\cap N_G(u_j)|-1$ possibilities for $f$.

Summing over $j$, possibly with repetitions, and adding the two incident-edge bounds gives
\begin{align*}
 &(r-1)+(d_G(u_1)-1)
 +\sum_{j=2}^{r}
   \bigl(|N_G(u_1)\cap N_G(u_j)|-1\bigr)\\
 &\qquad=
 d_G(u_1)-1+
 \sum_{j=2}^{r}|N_G(u_1)\cap N_G(u_j)|.\qedhere
\end{align*}
\end{proof}

\begin{proof}[Proof of \Cref{thm:deg-codeg-intro}]
For each edge $e=v_i u_1$, \cref{lem:later-conflict} gives
\[
\begin{aligned}
 d_G(u_1)-1+
 \sum_{j=2}^{r}|N_G(u_1)\cap N_G(u_j)|
 &\le \Delta-1+(r-1)\Dtwo(G)\\
 &\le \Delta-1+(d-1)\Dtwo(G).
\end{aligned}
\]
Hence the induced ordering of $V(\Gamma_B(G))=E(G)$ shows that $\Gamma_B(G)$ is $\bigl(\Delta-1+(d-1)\Dtwo(G)\bigr)$-degenerate.  Greedy coloring in reverse order therefore yields
\[
  \qB(G)=\chi(\Gamma_B(G))
  \le \Delta+(d-1)\Dtwo(G)
  \le d\Delta,
\]
where the final inequality follows from $\Dtwo(G)\le\Delta$.
\end{proof}

The same edge ordering also yields a stability statement near the bound $\qB(G)\le d\Delta$.

\begin{theorem}
\label{thm:near-core}
Let $d$, $\Delta$, and $s$ be integers satisfying
$1\le d\le\Delta$ and $0\le s<\Delta$.  Let $G$ be a
$d$-degenerate graph with $\Delta(G)\le\Delta$.  If
$\qB(G)\ge d\Delta-s$, then there is a vertex $v$ with distinct
neighbors $u_1,\ldots,u_d$ such that
\begin{equation}
\label{eq:deficiency-certificate}
 d\bigl(\Delta-d_G(u_1)\bigr)
 +\sum_{j=2}^{d}|N_G(u_1)\setminus N_G(u_j)|
 \le s.
\end{equation}
Consequently,
\[
  \left|\bigcap_{j=1}^{d}N_G(u_j)\right|\ge\Delta-s,
\]
and hence $G$ contains a copy of $K_{d,\Delta-s}$.
\end{theorem}

\begin{proof}
Let $q=\qB(G)$, and let $J$ be an induced subgraph of $\Gamma_B(G)$ of minimum order subject to $\chi(J)=q$.  Then $J$ is $q$-vertex-critical, so $\delta(J)\ge q-1\ge d\Delta-s-1$.

Fix a degeneracy ordering of $G$ and the induced edge ordering from
\cref{lem:later-conflict}.  Let $e=v_i u_1$ be the first member of
$V(J)\subseteq E(G)$ in this edge ordering, with $v_i$ preceding
$u_1$.  Every neighbor of $e$ in $J$ therefore follows $e$.  Write $u_1,\ldots,u_r$ for the later neighbors of $v_i$.  For
$2\le j\le r$, set
\[
  c_j=|N_G(u_1)\cap N_G(u_j)|.
\]By \cref{lem:later-conflict},
\[
  d\Delta-s-1
  \le d_J(e)
  \le d_G(u_1)-1+\sum_{j=2}^{r}c_j.
\]
Rearranging gives
\[
  (d-r)\Delta+\bigl(\Delta-d_G(u_1)\bigr)
  +\sum_{j=2}^{r}(\Delta-c_j)\le s.
\]
Every term on the left is nonnegative.  Since $r\le d$ and
$s<\Delta$, it follows that $r=d$.  Set $v=v_i$.  For
$2\le j\le d$,
\(
\Delta-c_j=\Delta-d_G(u_1)+|N_G(u_1)\setminus N_G(u_j)|.
\)
Thus the preceding inequality is precisely
\eqref{eq:deficiency-certificate}.

Let
\[
  C=\bigcap_{j=1}^{d}N_G(u_j).
\]
A union bound within $N_G(u_1)$ gives
\[
\begin{aligned}
 |C|
 &\ge d_G(u_1)
   -\sum_{j=2}^{d}|N_G(u_1)\setminus N_G(u_j)|\ge\Delta-s,
\end{aligned}
\]
where the last inequality follows from \eqref{eq:deficiency-certificate}, since $d\ge1$ and $\Delta-d_G(u_1)\ge0$.  Moreover,
$C\cap\{u_1,\ldots,u_d\}=\varnothing$, since $u_j\notin N_G(u_j)$ for every $j$.  Every vertex of $C$ is adjacent to every $u_j$, so $G$ contains a copy of $K_{d,|C|}$ and hence of $K_{d,\Delta-s}$.
\end{proof}

\begin{corollary}
\label{cor:deg-equality}
Let $d$ and $\Delta$ be integers with $1\le d\le\Delta$.  The maximum
possible value of $\qB(G)$ among all finite simple $d$-degenerate
graphs $G$ with $\Delta(G)\le\Delta$ is $d\Delta$.  Moreover, for
every such graph $G$, the following are equivalent:
\begin{enumerate}[label=\textup{(\roman*)}]
\item $\qB(G)=d\Delta$;
\item $G$ contains a copy of $K_{d,\Delta}$;
\item $G$ has at least $d$ vertices of degree $\Delta$ with the same
      open neighborhood.
\end{enumerate}
\end{corollary}

\begin{proof}
The upper bound follows from \cref{thm:deg-codeg-intro}.

If \textup{(i)} holds, then \cref{thm:near-core} with $s=0$ yields
\textup{(ii)}.  Conversely, any two edges of a copy of
$K_{d,\Delta}$ are either incident or opposite on a $4$-cycle.
Thus its $d\Delta$ edges are pairwise $B$-conflicting, so
\textup{(ii)}, together with the upper bound, implies \textup{(i)}.

Suppose that \textup{(ii)} holds, and let $A$ and $B$ be the parts
of such a copy, with $|A|=d$ and $|B|=\Delta$.  Every vertex of $A$
has all $\Delta$ vertices of $B$ as neighbors.  Since
$\Delta(G)\le\Delta$, each vertex of $A$ has open neighborhood
exactly $B$.  Hence \textup{(iii)} holds.

Conversely, suppose that \textup{(iii)} holds.  Choose $d$ such
vertices and let $X$ be the resulting set.  Let $Y$ be their common
open neighborhood.  Then $|Y|=\Delta$ and $X\cap Y=\varnothing$,
since $x\notin N_G(x)=Y$ for every $x\in X$.  Every vertex of $X$
is adjacent to every vertex of $Y$, so $G$ contains a copy of
$K_{d,\Delta}$.  Thus \textup{(ii)} holds.

Finally, $K_{d,\Delta}$ is $d$-degenerate and has maximum degree
$\Delta$.  Hence the value $d\Delta$ is attained.
\end{proof}

For sets $A$ and $B$, write $A\oplus B=(A\setminus B)\cup(B\setminus A)$.

\begin{corollary}
\label{cor:false-twin-stability}
Let $d$, $\Delta$, and $s$ be integers satisfying $1\le d\le\Delta$ and $0\le s<d$.  Let $G$ be a finite simple
$d$-degenerate graph with $\Delta(G)\le\Delta$ and $\qB(G)\ge d\Delta-s$.  Then some vertex $v$ has distinct neighbors
$u_1,\ldots,u_d$ such that $d_G(u_1)=\Delta$ and
\[
  \sum_{j=2}^{d}|N_G(u_1)\setminus N_G(u_j)|\le s,
  \qquad
  \sum_{j=2}^{d}|N_G(u_1)\oplus N_G(u_j)|\le2s.
\]
Consequently, at least $d-s$ of $u_1,\ldots,u_d$ have degree $\Delta$ and the same open neighborhood.  In particular, $G$
contains copies of $K_{d-s,\Delta}$ and $K_{d,\Delta-s}$.
\end{corollary}

\begin{proof}
Choose $v,u_1,\ldots,u_d$ as in \cref{thm:near-core}.  For
$2\le j\le d$, set
\[
  b_j=|N_G(u_1)\setminus N_G(u_j)|.
\]
By \eqref{eq:deficiency-certificate},
\[
  d\bigl(\Delta-d_G(u_1)\bigr)
  +\sum_{j=2}^{d}b_j
  \le s<d.
\]
Since all terms are nonnegative integers, this forces
$d_G(u_1)=\Delta$ and
\(
  \sum_{j=2}^{d}b_j\le s.
\)
This proves the first asserted inequality.

At most $s$ of $b_2,\ldots,b_d$ are positive.  If $b_j=0$, then
$N_G(u_1)\subseteq N_G(u_j)$.  Since
$d_G(u_1)=\Delta$ and $d_G(u_j)\le\Delta$, equality holds.  Thus,
together with $u_1$, at least $d-s$ of
$u_1,\ldots,u_d$ have degree $\Delta$ and the same open
neighborhood.

For $2\le j\le d$, the degree inequality
$d_G(u_j)\le d_G(u_1)$ gives
\[
  |N_G(u_j)\setminus N_G(u_1)|
  \le |N_G(u_1)\setminus N_G(u_j)|=b_j.
\]
Therefore
\[
  \sum_{j=2}^{d}|N_G(u_1)\oplus N_G(u_j)|
  \le 2\sum_{j=2}^{d}b_j
  \le 2s.
\]

The common open neighborhood of the resulting $d-s$ vertices has
size $\Delta$ and is disjoint from them, so it yields a copy of
$K_{d-s,\Delta}$.  The copy of $K_{d,\Delta-s}$ follows from
\cref{thm:near-core}.
\end{proof}
\color{black}

Together, \cref{thm:near-core,cor:false-twin-stability} prove
\cref{thm:near-core-intro}.

\begin{corollary}
\label{cor:one-defect}
Let $d$ and $\Delta$ be integers with $2\le d\le\Delta$, and let
$G$ be a finite simple $d$-degenerate graph with
$\Delta(G)\le\Delta$.  If $\qB(G)=d\Delta-1$, then there are sets
$X,Y\subseteq V(G)$ and a vertex $u^*\notin X$ such that
$|X|=d-1$, $|Y|=\Delta$, $N_G(x)=Y$ for every $x\in X$, and
\[
  |Y\setminus N_G(u^*)|=1,
  \qquad
  |N_G(u^*)\setminus Y|\le1.
\]
Thus $X$ is a set of pairwise false twins, each of degree $\Delta$.
In particular, $G$ contains a copy of $K_{d,\Delta-1}$.
\end{corollary}

\begin{proof}
Apply \cref{thm:near-core} with $s=1$, and let
$u_1,\ldots,u_d$ be the resulting vertices.  Set
$Y=N_G(u_1)$.  For $2\le j\le d$, set
\[
  b_j=|N_G(u_1)\setminus N_G(u_j)|.
\]
By \eqref{eq:deficiency-certificate},
\[
  d\bigl(\Delta-d_G(u_1)\bigr)
  +\sum_{j=2}^{d}b_j\le1.
\]
Since $d\ge2$, it follows that $d_G(u_1)=\Delta$ and
\(
  \sum_{j=2}^{d}b_j\le1.
\)

If every $b_j$ were zero, then
$Y\subseteq N_G(u_j)$ for every $j$.  Since $|Y|=\Delta$ and
$d_G(u_j)\le\Delta$, we would have $N_G(u_j)=Y$ for every $j$.
The vertices $u_1,\ldots,u_d$, together with $Y$, would then yield
a copy of $K_{d,\Delta}$, contradicting
$\qB(G)=d\Delta-1$.  Hence exactly one of
$b_2,\ldots,b_d$ equals $1$, and all the others are zero.

Let $u^*$ be the vertex corresponding to the unique positive
$b_j$, and set
\(
  X=\{u_1,\ldots,u_d\}\setminus\{u^*\}.
\)
For every $x\in X$, we have $Y\subseteq N_G(x)$.  The degree bound
therefore gives $N_G(x)=Y$.  Moreover,
\(
  |Y\setminus N_G(u^*)|=1,
\)
so $|Y\cap N_G(u^*)|=\Delta-1$.  Hence
\(
  |N_G(u^*)\setminus Y|
  =d_G(u^*)-(\Delta-1)\le1.
\)

Finally, $Y\cap N_G(u^*)$ consists of $\Delta-1$ common neighbors
of the $d$ vertices in $X\cup\{u^*\}$.  These two sets are disjoint
because $G$ is simple, and hence they form a copy of
$K_{d,\Delta-1}$.
\end{proof}

\begin{corollary}
\label{cor:Kdt-free}
Let $d$, $t$, and $\Delta$ be integers satisfying $1\le d\le\Delta$ and $1\le t\le\Delta$.  If $G$ is a finite simple $d$-degenerate graph with $\Delta(G)\le\Delta$ and contains no copy of $K_{d,t}$, then
$\qB(G)\le(d-1)\Delta+t-1.$
\end{corollary}

\begin{proof}
If the conclusion failed, then, since $\qB(G)$ is an integer, $\qB(G)\ge(d-1)\Delta+t=d\Delta-(\Delta-t)$.
Since $0\le\Delta-t<\Delta$, \cref{thm:near-core} with $s=\Delta-t$ would yield a copy of $K_{d,t}$ in $G$, a contradiction.
\end{proof}

\begin{corollary}
\label{cor:two-degenerate-second}
Let $\Delta\ge2$ be an integer.  If $G$ is a finite simple $2$-degenerate graph with $\Delta(G)\le\Delta$ and contains no copy
of $K_{2,\Delta}$, then $\qB(G)\le2\Delta-1$. This bound is sharp.
\end{corollary}

\begin{proof}
The upper bound follows from \cref{cor:Kdt-free} with $d=2$ and
$t=\Delta$.  For sharpness, let $H$ be the graph with vertices
$x,y,z_1,\ldots,z_{\Delta-1}$ and edges $xy,xz_i,yz_i$ for
$1\le i\le\Delta-1$.  The graph $H$ is $2$-degenerate, has maximum
degree $\Delta$, and has $2\Delta-1$ edges.  Any two nonincident
edges have the form $xz_i$ and $yz_j$ with $i\ne j$, and are
opposite on the $4$-cycle $x z_i y z_j x$.  Thus all edges of $H$
are pairwise $B$-conflicting, so $\qB(H)=2\Delta-1$.  Since $H$
has only $\Delta+1$ vertices, it contains no copy of
$K_{2,\Delta}$.
\end{proof}

For $3$-degenerate graphs, the bound from \cref{cor:Kdt-free} can be improved by one color.

\begin{theorem}
\label{thm:three-degenerate-second}
Let $\Delta\ge3$ be an integer.  If $G$ is a finite simple
$3$-degenerate graph with $\Delta(G)\le\Delta$ and contains no copy
of $K_{3,\Delta}$, then $\qB(G)\le3\Delta-2$.
\end{theorem}

\begin{proof}
Suppose for a contradiction that $\qB(G)\ge3\Delta-1$.  Put
$q=\qB(G)$, and let $J$ be a $q$-vertex-critical induced subgraph
of $\Gamma_B(G)$.  We identify $V(J)$ with the corresponding subset
of $E(G)$.  Then $\delta(J)\ge q-1\ge3\Delta-2$.

Fix a degeneracy ordering of $G$ and the associated edge ordering
from \cref{lem:later-conflict}.  Let $e_1=vu_1$ be the first member
of $V(J)$ in this edge ordering, where $v$ precedes $u_1$.  Write
$u_1,\ldots,u_r$ for the later neighbors of $v$, so $r\le3$.  Let
$h=|\{j\in\{1,\ldots,r\}:vu_j\in V(J)\}|$, and put
$d_i=d_G(u_i)$ and
$c_{ij}=|N_G(u_i)\cap N_G(u_j)|$ for $i\ne j$.

Every neighbor of $e_1$ in $J$ follows $e_1$.  Refining the proof
of \cref{lem:later-conflict} by counting only the $h-1$ other star
edges belonging to $V(J)$ gives
\[
\begin{aligned}
  3\Delta-2\le d_J(e_1)
  &\le(h-1)+(d_1-1)
       +\sum_{j=2}^{r}(c_{1j}-1)\\
  &\le r(\Delta-1)+h-1.
\end{aligned}
\]
If $r\le2$, then
$r(\Delta-1)+h-1\le r\Delta-1\le2\Delta-1<3\Delta-2$, a
contradiction.  Hence $r=3$, and the same bound gives $h\ge2$.

If $h=2$, then
\[
  3\Delta-2
  \le d_1+c_{12}+c_{13}-2
  \le3\Delta-2.
\]
Thus $d_1=c_{12}=c_{13}=\Delta$.  Since
$\Delta(G)\le\Delta$, we obtain
$N_G(u_1)=N_G(u_2)=N_G(u_3)$.  Because $G$ is simple, these three
vertices and their common neighborhood of size $\Delta$ form a
copy of $K_{3,\Delta}$, a contradiction.  Therefore $h=3$.

Set $e_i=vu_i$ for $i\in\{1,2,3\}$.  The block with earlier
endpoint $v$ consists of $e_1,e_2,e_3$, all of which belong to
$V(J)$.  Since $e_1$ is the first member of $V(J)$, every member of
$V(J)\setminus\{e_1,e_2,e_3\}$ follows this block.  Counting the
other two star edges separately and applying the preceding conflict
count to the remaining neighbors of $e_i$ gives
\[
  3\Delta-2\le d_J(e_i)
  \le2+(d_i-1)+(c_{ij}-1)+(c_{ik}-1)
  =d_i+c_{ij}+c_{ik}-1
\]
whenever $\{i,j,k\}=\{1,2,3\}$.

Let $a_i=\Delta-d_i$ and $b_{ij}=\Delta-c_{ij}$.  These are
nonnegative integers, with $b_{ij}=b_{ji}$, and the preceding
inequality is equivalent to
$a_i+b_{ij}+b_{ik}\le1$.  Since
$c_{ij},c_{ik}\le d_i$, we have $b_{ij},b_{ik}\ge a_i$, and hence
$3a_i\le1$.  Therefore $a_i=0$ for every $i$, so
$d_1=d_2=d_3=\Delta$ and
$b_{ij}+b_{ik}\le1$ whenever $\{i,j,k\}=\{1,2,3\}$.

Suppose that some $b_{ij}$ is positive.  By symmetry, let
$b_{12}>0$.  The inequalities for $i=1$ and $i=2$ give
$b_{12}=1$ and $b_{13}=b_{23}=0$.  Since all three vertices have
degree $\Delta$, it follows that
$N_G(u_1)=N_G(u_3)=N_G(u_2)$, contradicting $b_{12}=1$.  Hence
$b_{ij}=0$ for all distinct $i,j$.  Thus $u_1,u_2,u_3$ have the
same open neighborhood of size $\Delta$, again yielding a copy of
$K_{3,\Delta}$, a contradiction.
\end{proof}

The graph $K_{3,\Delta-1}$ satisfies the hypotheses, and its
$3\Delta-3$ edges are pairwise $B$-conflicting.  Hence the largest
possible value under the hypotheses of
\cref{thm:three-degenerate-second} is either $3\Delta-3$ or
$3\Delta-2$.  For $\Delta=3$, \cref{thm:main} gives the exact value
$6$.
\color{black}

\begin{remark}\label{rem:dense-endpoint}
Taking $d=\Delta$ in \cref{cor:deg-equality} shows that every finite simple graph $G$ with $\Delta(G)\le\Delta$ satisfies
$\qB(G)\le\Delta^2$, with equality if and only if $G$ contains a copy of $K_{\Delta,\Delta}$.  Under the degree bound, every vertex of such a copy already has degree $\Delta$, so the copy is a component of $G$.  Thus the equality characterization alone gives $\qB(G)\le\Delta^2-1$ when $G$ has no $K_{\Delta,\Delta}$ component.  For $\Delta\ge3$, \cref{thm:main} sharpens this to $\qB(G)\le\Delta(\Delta-1)$.
\end{remark}

\section{Proof of the sharp gap theorem}
\label{sec:sharp-gap-proof}

We prove \cref{thm:main} by reducing a smallest counterexample to a connected $\Delta$-regular simple graph and then treating the bipartite and nonbipartite cases separately.

\subsection{Parallel edges and star extensions}
\label{sec:star}

\begin{lemma}\label{lem:parallel-edge}
Let $\Delta\ge3$ be an integer, and put $k=\Delta(\Delta-1)$.  Let $G$ be a loopless multigraph with $\Delta(G)\le\Delta$, and let $e=uv$ be one of $m=\mu_G(u,v)\ge2$ parallel edges joining $u$ and $v$.  Every $B$-coloring of $G-e$ with colors from a set of size $k$ extends to a $B$-coloring of $G$ with colors from the same set.
\end{lemma}

\begin{proof}
Let $H=G-e$, and fix a $B$-coloring $\varphi$ of $H$.  Since $m\ge2$, some edge $e_0\in E(H)$ also joins $u$ and $v$.  Call an edge of $H$ forbidden for $e$ if it is incident with $u$ or $v$, or is opposite to $e$ on a $4$-cycle of $G$.

The number of edges of $H$ incident with $u$ or $v$ is
\[
\begin{aligned}
  (d_G(u)-1)+(d_G(v)-1)-(m-1)
  &=d_G(u)+d_G(v)-m-1 \\
  &\le 2\Delta-m-1.
  \end{aligned}
\]
Put $A=N_G(u)\setminus\{v\}$.  Every edge of $H$ opposite to $e$ on a $4$-cycle has an endpoint $x\in A$ and is not an edge joining $x$ to $u$.  Hence the number of such edges is at most
\[
\begin{aligned}
  \sum_{x\in A}\bigl(d_G(x)-\mu_G(x,u)\bigr)
 & \le |A|(\Delta-1) \\
 &  \le(\Delta-m)(\Delta-1),
\end{aligned}
\]
where $|A|\le d_G(u)-m\le\Delta-m$.  Possible repetitions in the sum only strengthen this upper bound.  Therefore the total number of forbidden edges is at most
\[
\begin{aligned}
  (2\Delta-m-1)+(\Delta-m)(\Delta-1)
  &=\Delta^2-(m-1)\Delta-1\le\Delta^2-\Delta-1=k-1.
\end{aligned}
\]
Choose a color $\alpha$ appearing on no forbidden edge, and extend $\varphi$ by setting $\varphi(e)=\alpha$.  The resulting coloring is proper.

It remains to check the $4$-cycles containing $e$.  If $C$ is such a cycle, then replacing $e$ by $e_0$ gives a $4$-cycle of $H$. Consequently, the three edges of $C-e$ have pairwise distinct colors.  Each is forbidden for $e$, so none has color $\alpha$. Thus $C$ is rainbow, and the extended coloring is a $B$-coloring of $G$.
\end{proof}

By \cref{lem:parallel-edge}, a smallest counterexample to \cref{thm:main} has no parallel edges.  Hence it suffices to
consider simple graphs throughout the rest of this subsection and in \cref{sec:bipartite,sec:clean}.

For an edge $uv\in E(G)$, put $A_{uv}=N_G(u)\setminus\{v\}$ and $B_{uv}=N_G(v)\setminus\{u\}$.  When $uv$ is fixed, we write simply $A$ and $B$.  Let $\calE_G(A_{uv},B_{uv})$ denote the set of edges with one endpoint in $A_{uv}$ and the other in $B_{uv}$, and put $\gG(uv)=|\calE_G(A_{uv},B_{uv})|$.
Thus $\gG(uv)$ is exactly the number of distinct edges opposite to $uv$ on a $4$-cycle.

Let $\kappaG(uv)$ denote the number of edges that $B$-conflict with $uv$.  Since an edge opposite to $uv$ is incident with neither endpoint of $uv$, $\kappaG(uv)=d_G(u)+d_G(v)-2+\gG(uv)$. Moreover,
\[
\begin{aligned}
  \gG(uv)& \le(d_G(u)-1)(d_G(v)-1), \\ 
  \kappaG(uv)& \le d_G(u)d_G(v)-1.
 \end{aligned}
\]

Deleting a single edge does not in general produce an extendable $B$-coloring.  Indeed, suppose that $uvabu$ is a $4$-cycle.  In a $B$-coloring of $G-uv$, the edges $va$ and $bu$, which are opposite on this cycle in $G$, may receive the same color.  No color assigned to $uv$ can then make the restored cycle rainbow. We therefore remove the entire star at one endpoint and restore all its edges simultaneously.

In a partial coloring, edges whose colors remain fixed are called \emph{old edges}, and uncolored edges are called \emph{new edges}. An old edge $f$ is an \emph{old conflict edge} for a new edge $e$ if $e$ and $f$ $B$-conflict in $G$.  For $u\in V(G)$, put $S_u=\{ux:x\in N_G(u)\}$.

\begin{lemma}\label{lem:star-recoloring}
Let $G$ be a simple graph, let $u\in V(G)$, and let $\mathcal C$ be a set of colors.  Suppose that $\varphi$ is a $B$-coloring of $G-S_u$ with colors from $\mathcal C$.  For each $x\in N_G(u)$, let $L_x\subseteq\mathcal C$ be the set of colors not used on any old conflict edge for $ux$. If the family $\{L_x\colon x\in N_G(u)\}$ admits a system of distinct representatives, then $\varphi$ extends to a $B$-coloring of $G$ with colors from $\mathcal C$.
\end{lemma}

\begin{proof}
Choose pairwise distinct colors $\gamma_x\in L_x$ and assign $\gamma_x$ to $ux$ for each $x\in N_G(u)$.  This extension is proper: the edges of $S_u$ receive distinct colors, and every old edge incident with $ux$ is incident with $x$ and has color different from $\gamma_x$.

Every $4$-cycle avoiding $u$ lies in $G-S_u$ and is already rainbow. Let $u x a y u$ be a $4$-cycle containing $u$.  The old edges $xa$ and $ay$ have distinct colors because they are incident, while the new edges $ux$ and $uy$ have distinct colors by the choice of the $\gamma_x$.  Moreover, $\gamma_x$ differs from both $\varphi(xa)$ and $\varphi(ay)$, since $xa$ is incident with $ux$ and $ay$ is opposite to $ux$.  Similarly, $\gamma_y$ differs from both old colors.  Thus the cycle is rainbow.
\end{proof}

\begin{lemma}\label{lem:low-degree-star}
Let $\Delta\ge3$ be an integer, put $k=\Delta(\Delta-1)$, and let $\mathcal C$ be a set of $k$ colors.  Let $G$ be a simple graph with $\Delta(G)\le\Delta$, and let $u\in V(G)$ with $t=d_G(u)\le\Delta-1$.  Every $B$-coloring of $G-S_u$ with colors
from $\mathcal C$ extends to a $B$-coloring of $G$ with colors from $\mathcal C$.

Moreover, if $uv\in E(G)$ and $G-uv$ has a $B$-coloring with colors from $\mathcal C$, then the edges $ux$, where
$x\in N_G(u)\setminus\{v\}$, may be uncolored and all edges of $S_u$ recolored to obtain a $B$-coloring of $G$ without changing the colors outside $S_u$.
\end{lemma}

\begin{proof}
The assertion is immediate when $t=0$, so assume $t\ge1$.  Fix $x\in N_G(u)$, and let $L_x$ be the set of colors available for $ux$ as in \cref{lem:star-recoloring}.  At most $d_G(x)-1$ old conflict edges are incident with $x$.  Every old edge opposite to $ux$ has the form $ay$ on a $4$-cycle $u x a y u$, where $a\in N_G(x)\setminus\{u\}$ and $y\in N_G(u)\setminus\{x\}$.  Hence there are at most $(d_G(x)-1)(t-1)$ such edges.  It follows that
\[
  |L_x|
  \ge k-t\bigl(d_G(x)-1\bigr)
  \ge(\Delta-t)(\Delta-1)
  \ge t,
\]
where the last inequality follows from $(\Delta-t)(\Delta-1)-t=\Delta(\Delta-1-t)\ge0$.

There are $t$ lists $L_x$, each of size at least $t$.  The union of any nonempty subfamily therefore has size at least the number of its members.  By Hall's theorem, the family $\{L_x\colon x\in N_G(u)\}$ admits a system of distinct representatives.  The result now follows from \cref{lem:star-recoloring}.

For the final assertion, start with a $B$-coloring of $G-uv$ and uncolor the edges $ux$ for $x\in N_G(u)\setminus\{v\}$.  Its restriction to $G-S_u$ is a $B$-coloring with colors from $\mathcal C$, so the first part recolors all edges of $S_u$.
\end{proof}

The regular case requires a sharper list-size estimate.  If $G$ is $\Delta$-regular, define the \emph{deficit} of an edge $uv$ by
$D(uv)=(\Delta-1)^2-\gG(uv)$.

\begin{lemma}\label{lem:clean-extension}
Let $\Delta\ge3$ be an integer, put $k=\Delta(\Delta-1)$, and let $\mathcal C$ be a set of $k$ colors.  Let $G$ be a $\Delta$-regular simple graph.  Suppose that adjacent vertices $u$ and $v$ satisfy $D(uv)\ge\Delta$ and, for every $x\in N_G(u)\setminus\{v\}$, $D(ux)\ge\Delta-1$. If $G-uv$ has a $B$-coloring with colors from $\mathcal C$, then $G$ has a $B$-coloring with colors from $\mathcal C$.
\end{lemma}

\begin{proof}
Let $\varphi$ be a $B$-coloring of $G-uv$ with colors from $\mathcal C$.  Uncolor the edges of $S_u\setminus\{uv\}$.  The
remaining coloring is a $B$-coloring of $G-S_u$.

For each $x\in N_G(u)$, the old conflict edges for $ux$ consist of the $\Delta-1$ edges incident with $x$ other than $ux$ and the $\gG(ux)$ edges opposite to $ux$ on $4$-cycles.  Hence
\[
  |L_x|
  \ge k-(\Delta-1)-\gG(ux)
  =(\Delta-1)^2-\gG(ux)
  =D(ux).
\]
Thus $|L_v|\ge\Delta$, while each of the other $\Delta-1$ lists has size at least $\Delta-1$.  To verify Hall's condition, let
$\varnothing\ne X\subseteq N_G(u)$.  If $v\in X$, then
\[
\left|\bigcup_{x\in X}L_x\right| \ge |L_v|\ge\Delta\ge |X|.
\]
If $v\notin X$, then $|X|\le\Delta-1$, and any $x\in X$ satisfies
\[
  \left|\bigcup_{y\in X}L_y\right|\ge |L_x|\ge\Delta-1\ge |X|.
\]
Thus Hall's condition holds, and \cref{lem:star-recoloring} completes the proof.
\end{proof}

\subsection{The bipartite case}\label{sec:bipartite}

\begin{theorem}\label{thm:bipartite}
Let $\Delta\ge4$ be an integer.  If $G$ is a finite simple bipartite graph with $\Delta(G)\le\Delta$ and no component isomorphic to $K_{\Delta,\Delta}$, then $\qB(G)\le\Delta(\Delta-1)$.
\end{theorem}

\begin{proof}
By K\H{o}nig's line-coloring theorem, partition $E(G)$ into
matchings $M_1,\ldots,M_\Delta$, some of which may be empty.   For each $i$, let
$H_i$ be the graph with vertex set $M_i$ in which two edges are
adjacent when they are opposite on a $4$-cycle of $G$.

Fix a bipartition $(L,R)$ of $G$ and an edge $e=xy\in M_i$, where
$x\in L$ and $y\in R$.  If $f=x'y'\in N_{H_i}(e)$, where
$x'\in L$ and $y'\in R$, then $xy',x'y\in E(G)$.  The map
$f\mapsto x'$ injects $N_{H_i}(e)$ into
$N_G(y)\setminus\{x\}$.  Hence
$d_{H_i}(e)\le d_G(y)-1\le\Delta-1$, and therefore
$\Delta(H_i)\le\Delta-1$.

Moreover, $H_i$ contains no $K_\Delta$.  Suppose otherwise that
$e_j=x_jy_j\in M_i$, $1\le j\le\Delta$, form a clique, where
$x_j\in L$ and $y_j\in R$.  Since $M_i$ is a matching, the vertices
$x_1,\ldots,x_\Delta$ are distinct, as are
$y_1,\ldots,y_\Delta$.  For $j\ne\ell$, adjacency of $e_j$ and
$e_\ell$ gives both $x_jy_\ell,x_\ell y_j\in E(G)$.  Thus these
vertices induce a copy of $K_{\Delta,\Delta}$.  Each of its vertices
already has degree $\Delta$ in this copy, so the maximum-degree
assumption makes the copy a component of $G$, a contradiction.

Let $C$ be a component of $H_i$.  If $C$ is complete, then
$|V(C)|\le\Delta-1$.  If $C$ is an odd cycle, then
$\chi(C)=3\le\Delta-1$.  In every other case, Brooks' theorem gives
$\chi(C)\le\Delta(C)\le\Delta-1$.  Hence
$\chi(H_i)\le\Delta-1$.

For each $i$, choose a proper coloring
$\psi_i\colon V(H_i)\to\{1,\ldots,\Delta-1\}$, and assign to every
$e\in M_i$ the color
\[
  \Phi(e)=(i,\psi_i(e)).
\]
This uses at most $\Delta(\Delta-1)$ colors.  Incident edges belong
to different matchings and hence receive different first
coordinates.  Opposite edges on a $4$-cycle either belong to
different matchings or are adjacent in some $H_i$ and receive
different second coordinates.  Thus every $4$-cycle is rainbow.
\end{proof}

\begin{remark}\label{rem:bipartite-sharp}
The bound in \cref{thm:bipartite} is attained by
$K_{\Delta,\Delta-1}$ and by the connected graph
$K_{\Delta,\Delta}-e$.  Indeed, all
$\Delta(\Delta-1)$ edges of $K_{\Delta,\Delta-1}$ pairwise
$B$-conflict.

For the second example, write its bipartition classes as
\[
  \{a_0,\ldots,a_{\Delta-1}\}
  \quad\text{and}\quad
  \{b_0,\ldots,b_{\Delta-1}\},
\]
and let $a_0b_0$ be the missing edge.  In the $B$-conflict graph,
the $(\Delta-1)^2$ edges $a_ib_j$ with $i,j\ge1$ form a clique
complete to all remaining vertices.  The edges incident with $a_0$
form a clique of order $\Delta-1$, as do those incident with $b_0$.
These two cliques are anticomplete, since a $4$-cycle having
$a_0b_j$ and $a_ib_0$ as opposite edges would require the missing
edge $a_0b_0$.  Therefore
\[
  \qB(K_{\Delta,\Delta}-e)
  =(\Delta-1)^2+(\Delta-1)
  =\Delta(\Delta-1).
\]
\end{remark}

\subsection{A clean endpoint in a regular nonbipartite graph}
\label{sec:clean}

\begin{theorem}\label{thm:clean-endpoint}
Let $\Delta\ge3$ be an integer, and let $G$ be a finite connected nonbipartite $\Delta$-regular simple graph.  Then there exist adjacent vertices $u$ and $v$ such that $D(uv)\ge\Delta$ and every $x\in N_G(u)\setminus\{v\}$ satisfies $D(ux)\ge\Delta-1$.
\end{theorem}

For the remainder of this subsection, fix $G$ and $\Delta$ as in \cref{thm:clean-endpoint}.  We begin by deriving a local identity for $D$ and translating it to the weighted false-twin quotient. The proof distinguishes the case in which every quotient class has size at least two from the case in which a singleton class exists; the latter is handled by a shortest-odd-cycle argument in the quotient.


\begin{lemma}\label{lem:local-identity}
Let $uv\in E(G)$, and put $A=N_G(u)\setminus\{v\}$, $B=N_G(v)\setminus\{u\}$, $T=A\cap B$.
Then
\[
  D(uv)=
  \sum_{x\in A}|N_G(x)\setminus N_G(v)|+e_G(T),
\]
where $e_G(T)$ denotes the number of edges with both endpoints in
$T$.
\end{lemma}

\begin{proof}
Since $G$ is $\Delta$-regular, $|A|=|B|=\Delta-1$.  The sum
$\sum_{x\in A}|N_G(x)\cap B|$ counts every edge of
$\calE_G(A,B)$ once and every edge with both endpoints in $T$ once
more.  Hence
\[
  \sum_{x\in A}|N_G(x)\cap B|
  =\gG(uv)+e_G(T).
\]

For $x\in A$, we have $u\in N_G(x)\cap N_G(v)$ and
$N_G(v)=B\cup\{u\}$.  Therefore
\[
  |N_G(x)\setminus N_G(v)|
  =\Delta-1-|N_G(x)\cap B|.
\]
Summing over $A$ gives
\[
\begin{aligned}
  \sum_{x\in A}|N_G(x)\setminus N_G(v)|
  &=(\Delta-1)^2-\sum_{x\in A}|N_G(x)\cap B|\\
  &=(\Delta-1)^2-\gG(uv)-e_G(T)\\
  &=D(uv)-e_G(T).
\end{aligned}
\]
Rearranging proves the identity.
\end{proof}


Recall that two distinct vertices are false twins if they have the same open neighborhood.  The equivalence classes under equality of open neighborhoods are called \emph{false-twin classes}.  Let $Q$ be the \emph{quotient} of $G$: its vertices are the false-twin classes, and distinct classes are adjacent in $Q$ if an edge of $G$ joins them.  Assign each $X\in V(Q)$ the positive integer weight $w(X)=|X|$.  We regard $(Q,w)$ as the weighted false-twin quotient of $G$.

Recall that two distinct vertices are false twins if they have the
same open neighborhood.  The equivalence classes under equality of
open neighborhoods are called \emph{false-twin classes}.  Let $Q$ be
the \emph{false-twin quotient} of $G$: its vertices are the
false-twin classes, and distinct classes are adjacent in $Q$ if an
edge of $G$ joins them.  Assign each $X\in V(Q)$ the positive integer
weight $w(X)=|X|$.  With this weighting, $Q$ is the weighted
false-twin quotient of $G$.

\begin{lemma}\label{lem:quotient-properties}
The false-twin quotient $Q$ has the following properties.
\begin{enumerate}[label=\textup{(\roman*)}]
  \item Every false-twin class is an independent set in $G$.
  \item Between two distinct false-twin classes there are either all
        possible edges or no edges.
  \item For every $X\in V(Q)$,
        \[
          \sum_{Y\in N_Q(X)}w(Y)=\Delta.
        \]
  \item Distinct vertices of $Q$ have distinct open neighborhoods
        in $Q$.
  \item The graph $Q$ is nonbipartite.
  \item The graph $Q$ is connected.
\end{enumerate}
\end{lemma}

\begin{proof}
If distinct vertices $x$ and $y$ in the same false-twin class were
adjacent, then
$y\in N_G(x)=N_G(y)$, contrary to the simplicity of $G$.  This
proves (i).

For (ii), suppose that $xy\in E(G)$, where $x\in X$ and $y\in Y$.
For any $x'\in X$, we have $y\in N_G(x')=N_G(x)$, so
$x'y\in E(G)$.  For any $y'\in Y$, we then have
$x'\in N_G(y')=N_G(y)$, so $x'y'\in E(G)$.  Thus all possible edges
between $X$ and $Y$ are present.

For $u\in X$, properties (i) and (ii) show that $N_G(u)$ is the
disjoint union of the classes in $N_Q(X)$.  Taking sizes and using
$d_G(u)=\Delta$ proves (iii).

For (iv), suppose that distinct vertices $P,R\in V(Q)$ satisfy
$N_Q(P)=N_Q(R)$.  They are not adjacent, since otherwise
$R\in N_Q(P)=N_Q(R)$, which is impossible in a simple graph.  Hence,
for any $p\in P$ and $r\in R$, properties (i) and (ii) give
\[
  N_G(p)
  =\bigcup_{Z\in N_Q(P)}Z
  =\bigcup_{Z\in N_Q(R)}Z
  =N_G(r).
\]
Thus $p$ and $r$ belong to the same false-twin class, contradicting
$P\ne R$.

If $Q$ were bipartite, the unions of the classes in its two parts
would form a bipartition of $G$, contrary to the hypothesis on $G$.
This proves (v).  Finally, by (i), every edge of $G$ joins two
distinct false-twin classes.  Hence every path in $G$ projects to a
walk in $Q$, and the connectedness of $G$ implies (vi).
\end{proof}

For $P,R\in V(Q)$, define their weighted neighborhood difference by
\[
  \eta(P,R)=
  \sum_{Z\in N_Q(P)\setminus N_Q(R)}w(Z).
\]
By \cref{lem:quotient-properties}(iii), every vertex of $Q$ has
weighted degree $\Delta$.  Hence, for all $P,R\in V(Q)$,
\begin{equation}
  \eta(P,R)
  =\Delta-
   \sum_{Z\in N_Q(P)\cap N_Q(R)}w(Z)
  =\eta(R,P).
  \label{eq:eta-symmetric}
\end{equation}
If $P\ne R$, then $N_Q(P)\ne N_Q(R)$ by
\cref{lem:quotient-properties}(iv).  Since these neighborhoods have
the same total weight, $N_Q(P)\setminus N_Q(R)$ is nonempty.
All weights are positive integers, so, for distinct
$P,R\in V(Q)$,
\begin{equation}
  \eta(P,R)\ge1.
  \label{eq:eta-positive}
\end{equation}

Swapping two vertices in the same false-twin class is an automorphism
of $G$.  Hence, for $XY\in E(Q)$, both $D(uv)$ and
$N_G(u)\cap N_G(v)$ are independent of the choice of representatives
$u\in X$ and $v\in Y$.  Denote them by $D(X,Y)$ and $T_{XY}$,
respectively.

\begin{lemma}\label{lem:quotient-formula}
For every edge $XY\in E(Q)$,
\[
  D(X,Y)=
  \sum_{P\in N_Q(X)\setminus\{Y\}}w(P)\eta(P,Y)
  +e_G(T_{XY}).
\]
Consequently,
\begin{equation}
  D(X,Y)\ge\Delta-w(Y)
  \quad\text{and}\quad
  D(X,Y)\ge\Delta-w(X).
  \label{eq:basic-D-bound}
\end{equation}
\end{lemma}

\begin{proof}
Choose representatives $u\in X$ and $v\in Y$, and put
$A=N_G(u)\setminus\{v\}$.  The set $A$ is the disjoint union of
$Y\setminus\{v\}$ and the classes
$P\in N_Q(X)\setminus\{Y\}$.

If $x\in Y\setminus\{v\}$, then $x$ and $v$ are false twins, so
$N_G(x)\setminus N_G(v)=\varnothing$.  If
$P\in N_Q(X)\setminus\{Y\}$ and $x\in P$, then
$N_G(x)\setminus N_G(v)$ is the disjoint union of the classes in
$N_Q(P)\setminus N_Q(Y)$.  Therefore
\[
  |N_G(x)\setminus N_G(v)|=\eta(P,Y).
\]
The formula now follows from \cref{lem:local-identity}.

For every $P\in N_Q(X)\setminus\{Y\}$,
\eqref{eq:eta-positive} gives $\eta(P,Y)\ge1$.  Hence
\[
  D(X,Y)\ge
  \sum_{P\in N_Q(X)\setminus\{Y\}}w(P)
  =\Delta-w(Y),
\]
where the equality follows from
\cref{lem:quotient-properties}(iii).  Applying this inequality with
$X$ and $Y$ interchanged and using $D(X,Y)=D(Y,X)$ gives the second
bound.
\end{proof}

For an ordered pair $(X,Y)$ of adjacent vertices of $Q$, define
\begin{equation}
  R(X,Y)=
  \sum_{P\in N_Q(X)\setminus\{Y\}}
  w(P)\bigl(\eta(P,Y)-1\bigr)+e_G(T_{XY}).
  \label{eq:R-def}
\end{equation}
Since $P\ne Y$ for every $P$ in the sum, all terms on the
right-hand side are nonnegative by \eqref{eq:eta-positive}.  Moreover,
\cref{lem:quotient-properties}(iii) gives
\[
  \sum_{P\in N_Q(X)\setminus\{Y\}}w(P)=\Delta-w(Y).
\]
Thus \cref{lem:quotient-formula} yields
\begin{equation}
  D(X,Y)=\Delta-w(Y)+R(X,Y),
  \qquad R(X,Y)\ge0.
  \label{eq:R-formula}
\end{equation}
The remainder need not be symmetric.  Comparing
\eqref{eq:R-formula} for $(X,Y)$ and $(Y,X)$ gives
\[
  R(X,Y)-R(Y,X)=w(Y)-w(X).
\]

A vertex $X\in V(Q)$ is \emph{clean} if $D(X,Y)\ge\Delta-1$ for every $Y\in N_Q(X)$.

\medskip

\begin{proof}[Proof of \cref{thm:clean-endpoint}]
It suffices to find a clean vertex $X$ having a neighbor $Y$ such
that $D(X,Y)\ge\Delta$.  Indeed, choose $u\in X$ and $v\in Y$.
By \cref{lem:quotient-properties}(ii), the vertices $u$ and $v$ are
adjacent, and
\[
  D(uv)=D(X,Y)\ge\Delta.
\]
If $x\in N_G(u)\setminus\{v\}$ belongs to the class $Z$, then
$Z\in N_Q(X)$, and hence
\[
  D(ux)=D(X,Z)\ge\Delta-1.
\]

Let $m=\min_{X\in V(Q)}w(X)$, the minimum size of a false-twin class.  We distinguish two cases.

\smallskip
\noindent\emph{Case 1: $m\ge2$.}
\smallskip

Choose $X\in V(Q)$ with $w(X)=m$, and let
$Y\in N_Q(X)$ be arbitrary.  By symmetry and
\cref{lem:quotient-formula}, applied to $(Y,X)$,
\[
  D(X,Y)=
  \sum_{P\in N_Q(Y)\setminus\{X\}}
    w(P)\eta(P,X)+e_G(T_{YX}).
\]
For each $P\in N_Q(Y)\setminus\{X\}$, we have $P\ne X$.
Since distinct quotient vertices have distinct neighborhoods of the
same total weight,
$N_Q(P)\setminus N_Q(X)$ is nonempty.  By the minimality of $m$,
it follows that $\eta(P,X)\ge m$.  Hence
\[
\begin{aligned}
  D(X,Y)
  &\ge m\sum_{P\in N_Q(Y)\setminus\{X\}}w(P)=m\bigl(\Delta-w(X)\bigr)
   =m(\Delta-m).
\end{aligned}
\]

Since $Q$ is nonbipartite, it contains an odd cycle.  A vertex on
this cycle has two distinct neighbors, each of weight at least $m$.
Thus $\Delta\ge2m\ge4$.  Consequently,
\[
  m(\Delta-m)\ge2(\Delta-2)\ge\Delta;
\]
the first inequality follows from $m(\Delta-m)-2(\Delta-2)=(m-2)(\Delta-m-2)\ge0$.
Since $Y$ was arbitrary, $D(X,Y)\ge\Delta$ for every
$Y\in N_Q(X)$. Choose $Y\in N_Q(X)$ and representatives $u\in X$ and $v\in Y$.
Every neighbor $x$ of $u$ belongs to some class
$Z\in N_Q(X)$, so
\[
  D(uv)=D(X,Y)\ge\Delta
  \quad \text{and} \quad
  D(ux)=D(X,Z)\ge\Delta
\]
for every $x\in N_G(u)\setminus\{v\}$.  Thus $u$ and $v$ satisfy
the conclusion of \cref{thm:clean-endpoint}.

\smallskip
\noindent\emph{Case 2: $m=1$.}
\smallskip

Let $\mathcal S=\{S\in V(Q)\colon w(S)=1\}.$
Every $S\in\mathcal S$ is clean, since
\eqref{eq:basic-D-bound} gives
$D(S,Y)\ge\Delta-w(S)=\Delta-1$ for every $Y\in N_Q(S)$. It remains to find an edge incident with a member of $\mathcal S$
whose deficit is at least $\Delta$.

Call an edge $SY\in E(Q)$ with $S\in\mathcal S$ a
\emph{tight singleton edge} if $D(S,Y)=\Delta-1$.  Suppose, to the
contrary, that no edge incident with a member of $\mathcal S$ has
deficit at least $\Delta$.  Since deficits are integers, every such edge is tight.  Thus, for
every $S\in\mathcal S$ and $Y\in N_Q(S)$,
\begin{equation}\label{eq:all-singleton-equality}
  D(S,Y)=\Delta-1.
\end{equation}

\begin{claim}\label{cl:propagation}
Let $S\in\mathcal S$ and $Y\in N_Q(S)$, and suppose that $D(S,Y)=\Delta-1$.  Then $e_G(T_{SY})=0$ and $\eta(P,S)=1$ for every $P\in N_Q(Y)\setminus\{S\}.$ Consequently, each of $N_Q(P)\setminus N_Q(S)$ and $N_Q(S)\setminus N_Q(P)$ consists of a single vertex of weight $1$.
\end{claim}

\smallskip
\noindent\emph{Proof of the claim.}
Since $D(Y,S)=D(S,Y)=\Delta-1$ and $w(S)=1$, \eqref{eq:R-formula}, applied to $(Y,S)$, gives $R(Y,S)=0$.
Every term in \eqref{eq:R-def} is nonnegative.  Hence $e_G(T_{SY})=0$ and $\eta(P,S)=1$ for every
$P\in N_Q(Y)\setminus\{S\}$.  Thus $N_Q(P)\setminus N_Q(S)$ has total weight $1$ and, since all weights
are positive integers, consists of a single vertex of weight $1$. By \eqref{eq:eta-symmetric}, the same holds for
$N_Q(S)\setminus N_Q(P)$. \hfill\(\lozenge\)

\bigskip

We first show that some shortest odd cycle of $Q$ contains a singleton.  Among the shortest odd cycles of $Q$, choose $C$ so that the length $t$ of a shortest path from $\mathcal S$ to $V(C)$ is minimum.  Let $P_0P_1\cdots P_t$ be such a path, where $P_0\in\mathcal S$ and $P_t\in V(C)$.  We show that $t=0$.

Suppose first that $t\ge3$.  By \eqref{eq:all-singleton-equality}, $P_0P_1$ is a tight singleton edge.  Applying \cref{cl:propagation} with $P=P_2$ shows that $N_Q(P_2)\setminus N_Q(P_0)$ consists of a single vertex of weight
$1$.  The vertex $P_3$ belongs to this difference, since an edge $P_0P_3$ would give a shorter path from $\mathcal S$ to $C$. Thus $P_3$ is a singleton closer to $C$ than $P_0$, a contradiction.

If $t=2$, the same application shows that $N_Q(P_2)\setminus N_Q(P_0)$ consists of a single vertex.  The two
neighbors of $P_2$ on $C$ are distinct, and neither is adjacent to $P_0$, since such an edge would give a path of length $1$ from $\mathcal S$ to $C$.  Both neighbors therefore belong to this difference, again a contradiction.

It remains to exclude $t=1$.  Write $P_1=X_0$ and $C=X_0X_1\cdots X_{\ell-1}X_0$, where $\ell$ is odd.  The vertex $P_0$ has no neighbor on $C$ other than $X_0$.  Indeed, suppose that $P_0X_j\in E(Q)$ for some $j\ne0$. The odd-length $X_0$--$X_j$ arc of $C$, together with $P_0X_0$ and $P_0X_j$, forms an odd cycle containing $P_0$.  The
other arc has positive even length and hence length at least $2$, so the new cycle has length at most $\ell$.  A shorter cycle contradicts the choice of $C$, while one of length $\ell$ contradicts the minimality of $t$.

The edge $P_0X_0$ is tight.  Applying \cref{cl:propagation} with $P=X_1$ shows that $N_Q(X_1)\setminus N_Q(P_0)$ consists of a single vertex of weight $1$.  Since $X_2$ belongs to this difference, it is a singleton on $C$, contradicting $t=1$.  Hence $t=0$.

Retain the cycle $C$ chosen above and write $C=X_0X_1\cdots X_{\ell-1}X_0$, with indices taken modulo $\ell$.  Thus $C$ is a shortest odd cycle of $Q$ and contains a singleton.  It is chordless, since any chord would produce a shorter odd cycle.

We next show that every vertex of $C$ is a singleton.  Suppose that $X_i$ is a singleton.  The edge $X_iX_{i-1}$ is tight, so
\cref{cl:propagation}, applied with $P=X_{i-2}$, shows that $N_Q(X_i)\setminus N_Q(X_{i-2})$ consists of a single vertex of
weight $1$.  This difference contains $X_{i+1}$.  Indeed, $X_{i+1}$ is adjacent to $X_i$ and is not adjacent to $X_{i-2}$:
when $\ell=3$, this follows because $X_{i+1}=X_{i-2}$ and $Q$ has no loops; when $\ell\ge5$, it follows from the chordlessness of $C$. Thus $X_{i+1}$ is a singleton.  Repeating the argument around $C$ shows that every $X_i$ is a singleton.

For each $i$, \eqref{eq:all-singleton-equality} and \eqref{eq:R-formula} now give
\begin{equation}\label{eq:R-zero-cycle}
  R(X_i,X_{i+1})=0.
\end{equation}

Suppose first that $\ell\ge5$.  In $R(X_i,X_{i+1})$, the term indexed by $P=X_{i-1}$ gives 
$$\eta(X_{i-1},X_{i+1})=1.$$  By
\eqref{eq:eta-symmetric}, the reverse difference also has total weight $1$.  Chordlessness gives
\[
\begin{aligned}
  X_{i-2}&\in N_Q(X_{i-1})\setminus N_Q(X_{i+1})  \quad \text{and} \quad X_{i+2}&\in N_Q(X_{i+1})\setminus N_Q(X_{i-1}).
  \end{aligned}
\]
Since all cycle vertices have weight $1$, it follows that
\[
\begin{aligned}
  N_Q(X_{i-1})\setminus N_Q(X_{i+1})
    &=\{X_{i-2}\} \quad \text{and} \quad
  N_Q(X_{i+1})\setminus N_Q(X_{i-1})
    &=\{X_{i+2}\}.
\end{aligned}
\]

Put $O_i=N_Q(X_i)\setminus\{X_{i-1},X_{i+1}\}$. The preceding equalities imply $O_{i-1}=O_{i+1}$.  Since $\ell$ is
odd, addition by $2$ visits every index modulo $\ell$, so all the
sets $O_i$ are equal to a common set $O$.  The weighted degree
identity gives
\[
  \sum_{Z\in O}w(Z)=\Delta-2>0.
\]
Choose $Z\in O$.  Since $C$ is chordless, $Z\notin V(C)$, and
$Z$ is adjacent to every vertex of $C$.  Hence
$ZX_iX_{i+1}Z$ is a triangle, contradicting the minimality of $C$.

Finally, suppose that $\ell=3$.  From
\eqref{eq:R-zero-cycle} and \eqref{eq:R-def}, for every $i$ we have
\[
  e_G(T_{X_iX_{i+1}})=0
  \quad\text{and}\quad
  \eta(X_{i-1},X_{i+1})=1.
\]
By \eqref{eq:eta-symmetric}, the reverse difference also has total
weight $1$.  Since $Q$ is simple and $C$ is a triangle,
\[
\begin{aligned}
  X_{i+1}&\in N_Q(X_{i-1})\setminus N_Q(X_{i+1}) \quad \text{and} \quad
  X_{i-1}&\in N_Q(X_{i+1})\setminus N_Q(X_{i-1}).
 \end{aligned}
\]
Therefore
\[
\begin{aligned}
  N_Q(X_{i-1})\setminus N_Q(X_{i+1})
    &=\{X_{i+1}\} \quad \text{and} \quad 
  N_Q(X_{i+1})\setminus N_Q(X_{i-1})
    &=\{X_{i-1}\}.
\end{aligned}
\]
Thus the three vertices $X_0,X_1,X_2$ have the same closed
neighborhood in $Q$.

The neighbors of $X_0$ outside $\{X_1,X_2\}$ have total weight
$\Delta-2>0$.  Choose
$Z\in N_Q(X_0)\setminus\{X_1,X_2\}$.  Equality of the closed
neighborhoods implies that $Z$ is adjacent to both $X_1$ and $X_2$.
Let $x_i$ be the unique vertex in $X_i$, and choose $z\in Z$.
By \cref{lem:quotient-properties}(ii), both $x_2$ and $z$ belong to
$T_{X_0X_1}$, and $x_2z\in E(G)$.  Hence
$e_G(T_{X_0X_1})>0$, contradicting
$e_G(T_{X_0X_1})=0$.

This contradiction shows that some edge $SY$, with
$S\in\mathcal S$, satisfies $D(S,Y)\ge\Delta$.  Choose
representatives $u\in S$ and $v\in Y$.  Since $S$ is clean, the edge $uv$ and every edge $ux$ with
$x\in N_G(u)\setminus\{v\}$ satisfy, respectively,
\[
  D(uv)\ge\Delta
  \quad\text{and}\quad
  D(ux)\ge\Delta-1.
\]
Thus $u$ and $v$ satisfy the conclusion of
\cref{thm:clean-endpoint}.
\end{proof}

\subsection{Completion of the proof}
\label{sec:mainproof}

\begin{proof}
Suppose that the theorem fails for some integer $\Delta\ge3$.  Put
$k=\Delta(\Delta-1)$ and fix a palette $\Omega$ of $k$ colors.
Choose, among all counterexamples, a graph $G$ with the minimum
number of edges and, subject to this, the minimum number of vertices.
Thus $G$ is a finite loopless multigraph with
$\Delta(G)\le\Delta$, no component of $G$ is isomorphic to the
simple graph $K_{\Delta,\Delta}$, and $G$ admits no $B$-coloring
from $\Omega$.

We use the following observation repeatedly.  If $F\subseteq E(G)$
and $H=G-F$, then every component of $H$ isomorphic to
$K_{\Delta,\Delta}$ is also a component of $G$.  Indeed, every vertex
$x$ of such a component satisfies $d_H(x)=\Delta$.  Since
$d_G(x)\le\Delta$, no edge of $F$ is incident with $x$, so the
component is unchanged in $G$.

The graph $G$ is connected.  Otherwise, each component $C$ has fewer
edges than $G$, unless it contains all edges of $G$, in which case it
has fewer vertices.  By minimality, after relabeling colors, every
component has a $B$-coloring from $\Omega$.  Combining these
colorings gives a $B$-coloring of $G$, a contradiction.

The graph $G$ is also simple.  If an edge $e$ belongs to a parallel
class of size at least two, then $G-e$ has no
$K_{\Delta,\Delta}$-component by the observation above.  By
minimality, $G-e$ has a $B$-coloring from $\Omega$, which extends to
$G$ by \cref{lem:parallel-edge}, again a contradiction.

If $\Delta=3$, the theorem of Xue, Hu, and Kong
\cite{XueHuKong2026} gives $\qB(G)\le6=k$, since $G$ is connected,
simple, and not isomorphic to $K_{3,3}$.  Hence we may assume that
$\Delta\ge4$.

We next show that $G$ is $\Delta$-regular.  Since $G$ is connected
and has an edge, every vertex has positive degree.  If $G$ is not
$\Delta$-regular, choose $u\in V(G)$ with
$1\le d_G(u)\le\Delta-1$, and put $H=G-S_u$.  The graph $H$ has
fewer edges than $G$ and, by the observation above, has no
$K_{\Delta,\Delta}$-component.  By minimality, $H$ has a
$B$-coloring from $\Omega$.  This coloring extends to $G$ by
\cref{lem:low-degree-star}, a contradiction.  Thus $G$ is
$\Delta$-regular.

If $G$ is bipartite, then \cref{thm:bipartite} gives
$\qB(G)\le\Delta(\Delta-1)=k$, a contradiction.  Hence $G$ is
nonbipartite.

By \cref{thm:clean-endpoint}, there are adjacent vertices $u$ and $v$ such that $D(uv)\ge\Delta$ and every
$x\in N_G(u)\setminus\{v\}$ satisfies $D(ux)\ge\Delta-1$. The graph $G-uv$ has fewer edges than $G$ and, by the observation above, has no $K_{\Delta,\Delta}$-component.  By minimality, it has a $B$-coloring from $\Omega$.  The displayed inequalities are precisely the hypotheses of \cref{lem:clean-extension}, so this coloring extends to $G$, a contradiction.
\end{proof}

\section{Concluding remarks and open problems}
\label{sec:conclusion}

For integers $\Delta\ge2$ and $1\le d\le\Delta$, let
$M_d(\Delta)$ denote the maximum of $\qB(G)$ over all finite simple
$d$-degenerate graphs $G$ with $\Delta(G)\le\Delta$ and containing
no copy of $K_{d,\Delta}$.  The lower bound is witnessed by
$K_{d,\Delta-1}$, while \cref{cor:Kdt-free} gives
\[
  d(\Delta-1)\le M_d(\Delta)\le d\Delta-1.
\]

The two bounds coincide when $d=1$, giving
$M_1(\Delta)=\Delta-1$, and
\cref{cor:two-degenerate-second} gives
$M_2(\Delta)=2\Delta-1$.  At the opposite endpoint, every copy of
$K_{\Delta,\Delta}$ in a graph of maximum degree at most $\Delta$
is a component. Thus, for every integer $\Delta\ge3$, \cref{thm:main} yields
$M_\Delta(\Delta)=\Delta(\Delta-1)$.
For $\Delta\ge4$, \cref{thm:three-degenerate-second} yields
\[
  3\Delta-3\le M_3(\Delta)\le3\Delta-2.
\]
Thus, among the cases $d<\Delta$, the general upper bound
$d\Delta-1$ is attained for $d=1,2$, but not for $d=3$.
For $4\le d<\Delta$, the case $s=1$ of
\cref{cor:one-defect} shows that if
$M_d(\Delta)=d\Delta-1$, then every graph attaining this value
contains copies of both $K_{d,\Delta-1}$ and
$K_{d-1,\Delta}$.

\begin{problem}
Determine $M_d(\Delta)$ for integers $3\le d<\Delta$.  In
particular, for each $\Delta\ge4$, decide which of the two values
$3\Delta-3$ and $3\Delta-2$ equals $M_3(\Delta)$.
\end{problem}

The sharp gap theorem raises a second extremal problem.

\begin{problem}
For each integer $\Delta\ge3$, characterize the finite connected
loopless multigraphs $G$ such that $\Delta(G)=\Delta$ and $\qB(G)=\Delta(\Delta-1)$.
\end{problem}

\section*{Acknowledgments}

This work was supported by NSFC (Nos. 12071048; 12422113). During the preparation of this manuscript, the authors used ChatGPT  to assist with English-language editing and expository organization.  All AI-assisted suggestions were reviewed and revised by the authors.  The authors assume responsibility for all content.

\section*{Data Availability Statement}

Data sharing is not applicable to this article, as no datasets were
generated or analyzed during the current study.

\end{document}